\documentclass[a4paper,11pt]{article}
\usepackage{latexsym}
\usepackage{amsmath}
\usepackage{amssymb}
\usepackage{enumerate}
\usepackage{amsthm}
\usepackage{array}
\usepackage{bm}
\usepackage{mathrsfs}
\newtheorem{theorem}{Theorem}[section]
\newtheorem{lemma}[theorem]{Lemma}
\newtheorem{corollary}[theorem]{Corollary}
\newtheorem{proposition}[theorem]{Proposition}
\theoremstyle{definition}

\newtheorem{remark}[theorem]{Remark}

\newtheorem*{annotation}{Annotation}

\DeclareMathOperator{\Ker}{Ker}

\DeclareMathOperator{\Tr}{Tr}
\DeclareMathOperator{\Nm}{Nm}
\DeclareMathOperator{\diag}{diag}

\newcommand{\gyokan}{\vskip 6pt}
\title{Relatives of the Hermitian curve}
\author{
Masaaki Homma
\\
Department of Mathematics\\
Kanagawa University\\
Yokohama 221-8686, Japan\\
homma@kanagawa-u.ac.jp
\and
Seon Jeong Kim
\\
 Department of Mathematics, and RINS\\
Gyeongsang National University\\
Jinju 660-701, Korea \\
skim@gnu.kr
}

\date{}

\begin{document}
\maketitle
\begin{abstract}
We introduce the notion of a relative of the Hermitian curve of degree $\sqrt{q}+1$
over $\mathbb{F}_q$, 
which is a plane curve defined by
\[(x^{\sqrt{q}}, y^{\sqrt{q}}, z^{\sqrt{q}})A {}^t \!(x,y,z) =0\]
with $A \in GL(3, \mathbb{F}_q)$,
and study their basic properties,
one of which is that the number of
$\mathbb{F}_q$-points of any relative of the Hermitian curve
of degree $\sqrt{q}+1$
is congruent to $1$ modulo $\sqrt{q}$.
\\
In the latter part of this paper,
we classify those curves having two or more rational inflexions.
\end{abstract}

\section{Introduction}
Throughout this paper, $q$ denotes an even power of a prime number $p$.
So $\sqrt{q}$ is also a power of $p$.
For $A \in GL(3, \mathbb{F}_q)$,
$A^{\ast}$ denotes the matrix ${}^t\!A^{(\sqrt{q})}$,
where the superscript $(\sqrt{q})$ indicates taking entry-wise
$\sqrt{q}$-th power,
and the superscripted prefix $t$ taking transpose.

The Hermitian curve is a plane curve defined by
\begin{equation}\label{definition_hermitian}
(x^{\sqrt{q}}, y^{\sqrt{q}}, z^{\sqrt{q}}) A \, {}^t\!(x,y,z) 
                                            =0,
\end{equation}
with the condition $A = A^{\ast}$

In this case, we say that $A$ satisfies the Hermitian condition.

Forgetting the Hermitian condition,
the equation (\ref{definition_hermitian}) still defines a nonsingular plane curves of degree $\sqrt{q}+1$ over $\mathbb{F}_q$,
which will be denoted by $C_A$.
We call such a curve a relative of the Hermitian curve, or shortly, a Hermitian-relative curve, and want to study them.

All Hermitian-relative curves, including Hermitian curves themselves, are projectively equivalent each other over 
the algebraic closure $\overline{\mathbb{F}}_q$ of $\mathbb{F}_q$,
which is due to Pardini~\cite[Proof of Proposition~3.7]{par1986}.
We will give a proof of this fact in the next section for the convenience of readers.

Any Hermitian-relative curve $C$ over $\mathbb{F}_q$
has a particular property on the number $N_q(C)$ of
$\mathbb{F}_q$-points:
\[
N_q(C) \equiv 1 \mod \sqrt{q}.
\]
The proof of this fact is also given in the following section.
\gyokan
Since the equation (\ref{definition_hermitian})
is homogeneous, the curve $C_A$ depends only on the image $\overline{A}$ of $A$
under the canonical map $GL(3, \mathbb{F}_q) \to PGL(3, \mathbb{F}_q)$.
By abuse of notation, $\overline{A}$ will be frequently denoted by $A$,
namely an element of $GL(3, \mathbb{F}_q)$
and its image in $PGL(3, \mathbb{F}_q)$ will be used interchangeably.
So the equation $A=B$ in $PGL(3, \mathbb{F}_q)$
means that there exists an element $\lambda \in \mathbb{F}_q^{\ast}$
so that $A=\lambda B$ in $GL(3, \mathbb{F}_q)$.

\begin{remark}
If $\overline{A} \in PGL(3, \mathbb{F}_q)$ satisfies the Hermitian condition,
that is to say, 
$\overline{A^{\ast}} = \overline{A}$,
then we can choose an element $A \in GL(3, \mathbb{F}_q)$
satisfying the Hermitian condition whose image
in $PGL(3, \mathbb{F}_q)$ is the assigned $\overline{A}$.
\end{remark}
\proof
Let $A$ be any source of $\overline{A}$.
Then there is an element $\lambda \in \mathbb{F}_q^{\ast}$
such that $A^{\ast}= \lambda A$.
Taking $\ast$-operation on the both side,
$A = \lambda^{\sqrt{q}} A^{\ast}$.
So $A= \lambda^{\sqrt{q}+1} A$,
which implies $\lambda^{\sqrt{q}+1} =1$.
Therefore there is an element $\rho \in \mathbb{F}_q$
so that $\rho^{1-\sqrt{q}} = \lambda$
(which is a special case of Hilbert's Theorem 90).
Then $\overline{\rho A} = \overline{A}$ and
$\rho A$ satisfies the Hermitian condition.
\qed

The polynomial of type (\ref{definition_hermitian})
is a linear combination of the monomials
\[
x^{\sqrt{q}+1}, y^{\sqrt{q}+1}, z^{\sqrt{q}+1},
x^{\sqrt{q}}y, y^{\sqrt{q}}z, z^{\sqrt{q}}x,
xy^{\sqrt{q}}, yz^{\sqrt{q}}, zx^{\sqrt{q}}
\]
over $\mathbb{F}_q$,
and vice versa.
So those polynomials forms an $\mathbb{F}_q$-vector space
of dimension $9$. Moreover the vector space is stable under the linear transformations on coordinates $x, y, z$.
We should note that
by a linear transformation
\[
\begin{pmatrix}
x \\ y\\ z
\end{pmatrix}
\mapsto
T^{-1}
\begin{pmatrix}
x \\ y\\ z
\end{pmatrix}
\]
with $T \in PGL(3, \mathbb{F}_q )$,
the polynomial
$
(x^{\sqrt{q}}, y^{\sqrt{q}}, z^{\sqrt{q}}) A \begin{pmatrix}
                                             x\\y\\z
                                            \end{pmatrix} 
$
goes to
\[
(x^{\sqrt{q}}, y^{\sqrt{q}}, z^{\sqrt{q}}) T^{\ast}A T\begin{pmatrix}
                                             x\\y\\z
                                            \end{pmatrix} .
\]
In other words, 
two relatives of the Hermitian curve $C_A$ and $C_B$,
are projectively equivalent over $\mathbb{F}_q$
if and only if there is a certain matrix $T \in PGL(3, \mathbb{F}_q)$
such that $B =T^{\ast}A T$ in $PGL(3, \mathbb{F}_q)$.
This is sometimes denoted by $A\sim B$.

\begin{remark}
Even if $\det A =0$, 
the curve defined by (\ref{definition_hermitian}) makes sense; 
however the curve is nonsingular if and only if $\det A \neq 0$.
Actually, the partial derivatives of
$f = (x^{\sqrt{q}}, y^{\sqrt{q}}, z^{\sqrt{q}}) A \begin{pmatrix}
                                             x\\y\\z
                                            \end{pmatrix} $
are given by
\begin{equation}\label{partial_derivatives}
(f_x, f_y, f_z) = (x^{\sqrt{q}}, y^{\sqrt{q}}, z^{\sqrt{q}}) A .
\end{equation}
\end{remark}
The properties mentioned above hold true
after any field extension $\mathbb{F}_q$.
\gyokan
As is explained in the next section,
the tangent line at a point of $C_A$ contacts
of multiplicity $\sqrt{q}$ in general,
but at some special points of multiplicity $\sqrt{q}+1$.
In the letter case, the point is called an inflexion.
In Section~3, we give the classification
up to $PGL(3, \mathbb{F}_q)$
of the curves $C_A$ with two or more inflexions
defined over $\mathbb{F}_q$.
In  the classification,
three types appear and
it can be summarized in Table~1.
\begin{table}[h]
\centering
\caption{Classification}
 \begin{tabular}{c|c|c|c}
    & $N_q(C_A)$ & inflexions & classes \\
  \hline
   (a) & $\sqrt{q}^3 +1$ & $\sqrt{q}^3 +1$ & 1\\
  \hline
  (b) & $\sqrt{q} +1$ & $\sqrt{q}+1$ & $\sqrt{q}$ \\
  \hline
  (c) & $q+1$ & $2$ & $\frac{1}{2}(\sqrt{q}+1)(\sqrt{q}-2)$
 \end{tabular}
\end{table}
In the table, the first column indicates the label of each type which agrees
with the item in Theorem~\ref{maintheorem},
the second the number of $\mathbb{F}_q$-points of $C_A$
belongs the assigned type,
the third the number of inflexions,
and the fourth is the number of $\mathbb{F}_q$-equivalent
classes in each type.

\begin{annotation}
$\bullet$ Although a homogeneous coordinates of a point has an ambiguity by nonzero multiple, we understand all values of the coordinates of
any $\mathbb{F}_q$-point are elements in $\mathbb{F}_q$.
Moreover, we use two words ``$\mathbb{F}_q$-point"
and ``rational point" interchangeably.
\\
$\bullet$ The set of $\mathbb{F}_q$-points of a curve $C$ is denoted by
$C(\mathbb{F}_q)$.
\\
$\bullet$ For a field, the multiplicative group of nonzero elements
is marked by the superscript $\ast$.
\\
$\bullet$ The trace map from $\mathbb{F}_q$ to $\mathbb{F}_{\sqrt{q}}$,
that is,
$
\mathbb{F}_q \ni \alpha \mapsto \alpha +\alpha^{\sqrt{q}} 
\in \mathbb{F}_{\sqrt{q}},
$
and the norm map $\mathbb{F}_q^{\ast}$ to $\mathbb{F}_{\sqrt{q}}^{\ast}$,
that is,
$
\mathbb{F}_q^{\ast} \ni \alpha \mapsto 
\alpha^{\sqrt{q}+1} \in \mathbb{F}_{\sqrt{q}}^{\ast}
$
are simply denoted by $\Tr$ and $\Nm$ respectively.
\\
$\bullet$ Unless we specify the field of definition, 
`projectively equivalent' means `projectively equivalent
over $\mathbb{F}_q$'.
\\
$\bullet$ Sometimes a curve defined by $g=0$ will be indicated by $\{g=0\}$.
\end{annotation}

\section{Basic properties}
We shall start this section with Pardini's observation \cite{par1986},
which is our motivation to study
the Hermitian-relative curve.

\begin{theorem}\label{pardini}
Let $f(x,y,z) \in \mathbb{F}_q [x,y,z]$ be a homogeneous polynomial of
degree $\sqrt{q}+1$.
The curve $C$ defined by $f(x,y,z)=0$ in $\mathbb{P}^2$ is
projectively equivalent over $\overline{\mathbb{F}}_q$
to a Hermitian curve of degree $\sqrt{q}+1$
if and only if $C$ is a Hermitian-relative curve,
that is,
there exists $A \in PGL(3, \mathbb{F}_q)$ such that
\[
f(x,y,z) = (x^{\sqrt{q}}, y^{\sqrt{q}}, z^{\sqrt{q}}) A {}^t\!(x,y,z).
\]
\end{theorem}
\proof (Pardini)
First suppose that $C$ is projectively equivalent
to 
\[
g(x,y,z)=x^{\sqrt{q}+1}+ y^{\sqrt{q}+1} + z^{\sqrt{q}+1} =0
\]
over $\overline{\mathbb{F}}_q$.
Then there is a matrix
$T = (t_{i,j}) \in GL(3, \overline{\mathbb{F}}_q)$
so that
\[
f(x,y,z) = g((x,y,z){}^tT),
\]
that is,
\begin{align*}
f(x,y,z) &=  
\sum_{j=1}^{3} \left( t_{j1}x + t_{j2}y +t_{j3}z  \right)^{\sqrt{q} +1} \\
&= (x^{\sqrt{q}}, y^{\sqrt{q}}, z^{\sqrt{q}}) T^{\ast} T
                                           \begin{pmatrix}
                                             x\\y\\z
                                            \end{pmatrix}.
\end{align*}
Note that each entry of $ T^{\ast} T$
gives a coefficient of distinct monomial of $f(x,y,z)$.
So $ T^{\ast} T \in PGL(3, \mathbb{F}_q)$.

To see the converse, let us consider a morphism
associated with an assigned $A \in PGL(3, \mathbb{F}_q)$:
\[
\Phi_A : PGL(3, \overline{\mathbb{F}}_q) \ni T
\mapsto  T^{\ast}A T \in PGL(3, \overline{\mathbb{F}}_q).
\]
If $\Phi_A(S) = \Phi_A(T)$ for $S, T \in PGL(3, \overline{\mathbb{F}}_q)$,
then
\[
 (ST^{-1})^{\ast} A (ST^{-1}) =A
\]
which means $ST^{-1}$ gives a linear automorphism of $C_A$
over $\overline{\mathbb{F}}_q$
in $\mathbb{P}^2$.

The automorphism group of $C_A$ over $\overline{\mathbb{F}}_q$
is finite if $\sqrt{q} >2$, because the genus of $C_A$
is grater than $1$.
When $\sqrt{q} =2$, $C_A$ is an elliptic curve over $\overline{\mathbb{F}}_q$
in $\mathbb{P}^2$.
The linear automorphism group of any elliptic curve in $\mathbb{P}^2$
is also finite.
Therefore the image of $\Phi_A$ is a constructible set 
(which is a theorem of Chevalley \cite[Corollary~2 in I \S 8]{mum1999}) 
of dimension $8$
in the target space $ PGL(3, \overline{\mathbb{F}}_q)$.
Therefore two images of $\Phi_A$ and $\Phi_B$ meet in 
$PGL(3, \overline{\mathbb{F}}_q)$ for any two $A, B \in PGL(3, \mathbb{F}_q)$,
that is, there is an element $T \in PGL(3, \overline{\mathbb{F}}_q)$
such that $ T^{\ast}A T  =B$.
In particular, if one takes as $B =I_3$,
one gets a projective transformation over $\overline{\mathbb{F}}_q$
which sends $C_A$ to a Hermitian curve.
\qed

\gyokan

In the next corollary, 
$T_P(C)$ denotes the (embedded) tangent line to a plane curve $C$ at $P$,
and $i(C.T_P(C); P)$ the intersection multiplicity of $C$ and $T_P(C)$ at $P$.
\begin{corollary}\label{characterization}
Suppose that $q >4$.
Among nonsingular plane curves of degree $\sqrt{q}+1$,
a Hermitian-relative curve $C$ is characterized by the property:
for any  $\overline{\mathbb{F}}_q$-point $P$ of $C$,
$i(C.T_P(C); P)$ is
either $\sqrt{q}$ or $\sqrt{q}+ 1$.
\end{corollary}
\proof
This characterization comes from \cite[Theorem~6.1, and Theorem~3.5]{hom1987}
together with Theorem~\ref{pardini}
\qed

The following is the main purpose of this section.

\begin{theorem}\label{number_mod_sqrt_q}
Let $C$ be a relative of the Hermitian curve of degree $\sqrt{q}+1$
over $\mathbb{F}_q$. Then
\[
N_q(C) \equiv 1 \mod \sqrt{q}.
\]
In particular, the set $C(\mathbb{F}_q)$ of $\mathbb{F}_q$-points 
is not empty.
\end{theorem}
\proof
Let $g = \frac{1}{2}\sqrt{q}(\sqrt{q}-1)$,
which is the genus of any Hermitian-relative curve.
Let $C'$ be a Hermitian curve of degree $\sqrt{q}+1$,
e.g., $C'$ is one defined by 
$
x^{\sqrt{q}+1} +y^{\sqrt{q}+1} +z^{\sqrt{q}+1} =0.
$
By Theorem~\ref{pardini}, $C$ and $C'$ are projectively
equivalent over some extension field $\mathbb{F}_{q^{s}}$.
Let
\[
Z_C(t) = \prod_{j=1}^{2g}(1-\alpha_j t)/(1-t)(1-qt)
\]
be the zeta function of $C$ over $\mathbb{F}_q$.
By Weil's theorem,
\begin{equation}\label{weil}
N_{q^s}(C) = 1 + q^s - \sum_{j=1}^{2g} \alpha_j^s
\text{\ and \ }
\alpha_j = \sqrt{q}\eta_j
\text{\ with \ }
|\eta_j| =1.
\end{equation}
On the other hand,
since $C'$ is a maximal curve over $\mathbb{F}_q$,
its zeta function is
\[
Z_{C'}(t) = \left(1- (-\sqrt{q})t \right)^{2g}/(1-t)(1-qt).
\]
Since $C$ and $C'$ are isomorphic over $\mathbb{F}_{q^s}$,
\[
\{\alpha_1^s, \dots , \alpha_{2g}^s \}
= \{ \overbrace{(-\sqrt{q})^s, \dots , (-\sqrt{q})^s}^{2g} \}.
\]
(As for this classical theory, one may consult
\cite[Chapter~5]{sti2008}.)

Since we may choose $s$ to be even, $\eta_j^s =1$ for each $j = 1, \dots , 2g$.
In particular, $\eta_j$ is an algebraic integer,
and so is $\sum_{j=1}^{2g} \eta_j$.
Going back to the formula (\ref{weil}) (with $s=1$),
\[
\sum_{j=1}^{2g} \eta_j = \frac{q+1 - N_q(C)}{\sqrt{q}}.
\]
Since $\sqrt{q}$ is a rational integer by our assumption,
$\sum_{j=1}^{2g} \eta_j$ is a rational number.
Therefore it is a rational integer.
Hence
\[
N_q(C) = q+1 -\sqrt{q}(\sum_{j=1}^{2g} \eta_j) \equiv 1 \mod \sqrt{q}.
\]
This completes the proof.
\qed

\begin{corollary}\label{Cor_number_mod_sqrt_q}
For a relative of the Hermitian curve $C$ over $\mathbb{F}_q$,
there is a nonnegative integer $m =m_C$
with
\[
m =q \text{\ or \ } \sqrt{q}+2 \geq m \geq 0
\]
so that
$N_q(C) = m\sqrt{q} +1$.
Moreover
$m_C = q$ if and only if
either
$C$ is Hermitian, or
$q=4$ and
$C$ is projectively equivalent to
the curve
\[
x^{\sqrt{q}+1} + \omega y^{\sqrt{q}+1} +\omega^2 z^{\sqrt{q}+1} =0.
\]
over $\mathbb{F}_4$,
where
$\omega \in \mathbb{F}_4 \setminus \mathbb{F}_2$.
\end{corollary}
\proof
The additional statement is somewhat classical
\cite{hir-sto-tha-vol1991, ruc-sti1994, hom2024}.
For the first statement,
when $q=4$, there is nothing new from Theorem~\ref{number_mod_sqrt_q}.
Hence we assume that
$q>4$.
Let $C = C_A$, and $(x_0, y_0, z_0) \in C(\mathbb{F}_q)$.
Then
$
(x_0^{\sqrt{q}}, y_0^{\sqrt{q}}, z_0^{\sqrt{q}}) A {}^t\!(x_0,y_0,z_0)
=0.
$
Taking $\sqrt{q}$-th power, and noting $x_0^q= x_0, y_0^q= y_0, z_0^q= z_0$,
we know that
$(x_0, y_0, z_0) \in C_{A^{\ast}}$.
If $C$ is not Hermitian, $C_A$ and $ C_{A^{\ast}}$
are distinct irreducible curves.
Hence the intersection of these two curves consists of at most
$(\sqrt{q}+1)^2$ points, which gives an upper bound for $N_q(C)$.
\qed

\gyokan

Before closing this section, we shall digress a little to explain
a couple of geometric properties of relatives of the Hermitian curve.

\begin{proposition}
For $A \in GL(3, \mathbb{F}_q)$,
the dual curve of $C_A$ is 
$C_{{}^t \! A^{-1}}$.
\end{proposition}
\proof
Let $u,v,w$ be the dual coordinate of the dual projective plane
with point-line incidence $ux+vy+wz=0$.
By (\ref{partial_derivatives}),
the equation after eliminating $x, y, z$ from
\begin{equation}\label{dual_equation}
(u,v,w)= (x^{\sqrt{q}}, y^{\sqrt{q}}, z^{\sqrt{q}}) A
\end{equation}
and (\ref{definition_hermitian})
is an equation of the dual curve.
Taking entry-wise $\sqrt{q}$-th power of (\ref{definition_hermitian}),
\[
(x^q, y^q, z^q) A^{(\sqrt{q})} \begin{pmatrix}
                                      x^{\sqrt{q}}\\y^{\sqrt{q}}\\z^{\sqrt{q}}
                                            \end{pmatrix}=0,
\]
which can be written
\[
(u^{\sqrt{q}}, v^{\sqrt{q}}, w^{\sqrt{q}}) {}^t\!A^{-1} \begin{pmatrix}
                                             u\\v\\w
                                            \end{pmatrix} =0
\]
by
(\ref{dual_equation}).
This completes the proof.
\qed

\gyokan

For a Hermitian-relative curve $C_A$, the curve $C_{A^{\ast}}$
is called the mirror curve of $C_A$.
Of course, the mirror curve of $C_{A^{\ast}}$
is the original one $C_A$, and
$C_A$ is Hermitian if and only if $C_{A^{\ast}}=C_A$.

\begin{proposition}
For a mirror pair $(C_A, C_{A^{\ast}})$,
the following properties hold.
\begin{enumerate}[{\rm (1)}]
\item $C_A(\mathbb{F}_q)= C_{A^{\ast}}(\mathbb{F}_q)$.
\item For $P \in C_A(\mathbb{F}_q)= C_{A^{\ast}}(\mathbb{F}_q)$,
$P$ is an inflexion of $C_A$ if and only if
an inflexion of $C_{A^{\ast}}$.
In this case, $T_P(C_A) = T_P(C_A^{\ast})$.
\item For a non-inflexion $P \in C_A(\mathbb{F}_q)$,
then there is a point $P' \in C_A(\mathbb{F}_q)$ so that
$C_A.T_PC_A = \sqrt{q}P + P'$.
In this case, $C_{A^{\ast}} .T_{P'}C_{A^{\ast}}= \sqrt{q}P' + P$.
In particular $P'$ is a non-inflexion of $C_{A^{\ast}}$.
\end{enumerate}
\end{proposition}
\begin{proof}
If $C_A$ is Hermitian, then the assertions (1) and (2) are trivially true,
and (3) is nonsense.
So $C_A$ may be assumed to be non-Hermitian.

We already saw the assertion (1) in
the proof of Corollary~\ref{Cor_number_mod_sqrt_q}.

(2) Since $(T^{\ast}AT)^{\ast} = T^{\ast}A^{\ast}T$
for any $T \in PGL(3, \mathbb{F}_q)$,
we may assume that
$P=(0,0,1)$ and $T_PC_A= \{y=0\}$.
Let $A= (a_{ij})$.
Since $P\in C_A$, $a_{33}=0$.
Hence an affine equation of $C_A$ on $\{ z \neq 0\}$
is given by
\[
a_{11}x^{\sqrt{q}+1}
+(a_{12}y + a_{13})x^{\sqrt{q}}
+(a_{21}x + a_{22}y + a_{23})y^{\sqrt{q}}
+a_{31}x +  a_{32}y =0.
\]
Since $T_PC_A = \{y=0 \}$
and the origin $P$ is an inflexion of $C_A$,
we have
$a_{31} =0$, $a_{13} =0$ and $a_{11}\neq 0$.
Hence we may assume that $a_{11}=1$.
To sum up,
\[
A=
\begin{pmatrix}
1 & a_{12} & 0\\
a_{21}& a_{22}&a_{23}\\
0 & a_{32}&0
\end{pmatrix}.
\]
Hence
\[
A^{\ast}=
\begin{pmatrix}
1 & a_{21}^{\sqrt{q}} & 0\\
a_{12}^{\sqrt{q}}& a_{22}^{\sqrt{q}}&a_{32}^{\sqrt{q}}\\
0 & a_{23}^{\sqrt{q}}&0
\end{pmatrix},
\]
in other words, an affine equation of $C_{A^{\ast}}$ is
\[
x^{\sqrt{q}+1} 
+a_{21}^{\sqrt{q}}y x^{\sqrt{q}}
+(a_{12}^{\sqrt{q}}x + a_{22}^{\sqrt{q}}y + a_{32}^{\sqrt{q}})y^{\sqrt{q}}
+  a_{32}^{\sqrt{q}}y =0.
\]
This shows that the origin is also an inflexion of $C_{A^{\ast}}$
with tangent line $\{ y=0\}$.

(3) By Corollary~\ref{characterization},
$C_A.T_PC_A = \sqrt{q}P + P'$ for some point $P'$.
Since two divisors $C_A.T_PC_A$ and $\sqrt{q}P$
are defined over $\mathbb{F}_q$,
so is the residual one $P'$.
Hence $P' \in C_A(\mathbb{F}_q) = C_{A^{\ast}}(\mathbb{F}_q)$.
Let $P= (x_0, y_0, z_0)$
and $P'= (x_1, y_1, z_1)$.
Then
\[
(x_0^{\sqrt{q}}, y_0^{\sqrt{q}}, z_0^{\sqrt{q}})
A {}^t \!(x_1, y_1, z_1)
\]
by Remark~\ref{partial_derivatives}.
Hence
\[
(x_1^{\sqrt{q}}, y_1^{\sqrt{q}}, z_1^{\sqrt{q}})
A^{\ast} {}^t \!(x_0, y_0, z_0),
\]
which means
$P \in C_{A^{\ast}}. T_{P'} C_{A^{\ast}}$,
and hence
$C_{A^{\ast}}. T_{P'} C_{A^{\ast}} = \sqrt{q}P' + P.$
\end{proof}

\section{Hermitian-relative curves with two or more rational inflexions}
The purpose of this section is to give the classification of
the Hermitian-relative curves having two or more rational inflexions,
which is summarized in the following theorem.

\begin{theorem}\label{maintheorem}
Let $C$ be a relative of the Hermitian curve.
Suppose that $C$ has at least two rational
inflexions.
Then $C$ is projectively equivalent to
\begin{equation}\label{at_least_two_inflexions}
x^{\sqrt{q}}y +\omega xy^{\sqrt{q}} + z^{\sqrt{q}+1} =0
\end{equation}
with $\omega \in \mathbb{F}_q^{\ast}$.
Conversely a curve defined by $(\ref{at_least_two_inflexions})$
has at least two rational inflexions.
\begin{enumerate}[{\rm (a)}]
\item If $\omega =1$, then the curve $(\ref{at_least_two_inflexions})$
has $\sqrt{q}^3 +1$ rational points, and all of them are inflexions;
and it is Hermitian.
Therefore a Hermitian-relative curve with $\sqrt{q}^3 +1$ rational inflexions
is unique up to projective equivalence.
\item If $\Nm \omega =1$ and $\omega \neq 1$,
then the curve $(\ref{at_least_two_inflexions})$
has $\sqrt{q}+1$ rational points, and all of them are inflexions.
In this case, the curve is
projectively equivalent to
\begin{equation}\label{diagonaloftype1}
x^{\sqrt{q}+1} + y^{\sqrt{q}+1} + \eta z^{\sqrt{q}+1} =0
\end{equation}
with a certain $\eta \in \mathbb{F}_q^{\ast} 
\setminus \mathbb{F}_{\sqrt{q}}^{\ast}.$
Conversely, a curve $(\ref{diagonaloftype1})$
with $\eta \in \mathbb{F}_q^{\ast} 
\setminus \mathbb{F}_{\sqrt{q}}^{\ast}$
is projectively equivalent to
a curve $(\ref{at_least_two_inflexions})$
for a certain $\omega$ with $\Nm \omega =1$ and $\omega \neq 1$.

For two curves
$x^{\sqrt{q}+1} + y^{\sqrt{q}+1} + \eta z^{\sqrt{q}+1} =0$
and
$x^{\sqrt{q}+1} + y^{\sqrt{q}+1} + \eta' z^{\sqrt{q}+1} =0$
with $\eta, \, \eta' \in \mathbb{F}_q^{\ast}
 \setminus \mathbb{F}_{\sqrt{q}}^{\ast}.$
are projectively equivalent to each other if and only if 
$\eta \eta'^{-1} \in \mathbb{F}_{\sqrt{q}}^{\ast}$.
Therefore, there are
exactly $\sqrt{q}$ non-equivalent Hermitian-relative curves,
each of which has
$\sqrt{q}+1$ rational points
and all of them are inflexions.
\item If $\Nm \omega \neq 1$, then the curve $(\ref{at_least_two_inflexions})$
has $q+1$ rational points, and only two of them are inflexions.
For two curves $x^{\sqrt{q}}y +\omega xy^{\sqrt{q}} + z^{\sqrt{q}+1} =0$
and $x^{\sqrt{q}}y +\omega' xy^{\sqrt{q}} + z^{\sqrt{q}+1} =0$
with $\omega, \, \omega' \in \mathbb{F}_q^{\ast} \setminus \Ker( \Nm )$
are projectively equivalent if and only if
$\omega' = \omega\ \text{or}\ \omega^{-\sqrt{q}}$.
Therefore, there are exactly $\frac{1}{2}(\sqrt{q}+1)(\sqrt{q}-2)$
non-equivalent Hermitian-relative curves with just two rational inflexions.
\end{enumerate}
\end{theorem}

\subsection{Some equations}
Since any trace map and norm map between finite fields
are surjective,
it is easy to show Hilbert's Theorem 90 in our case:

\begin{lemma}\label{lemma_hilbert90}
There are two exact sequences{\rm :}
\begin{equation}\label{hilbert90_addtive}
\begin{array}{rcccccccl}
0\to & \mathbb{F}_{\sqrt{q}} & \stackrel{i}{\to}&  \mathbb{F}_q &
 \to & \mathbb{F}_q & 
 \xrightarrow{\Tr} & \mathbb{F}_{\sqrt{q}} &\to 0 \\
     &  &  &     \alpha &\mapsto &\alpha -\alpha^{\sqrt{q}}  &
     & &
\end{array}
\end{equation}
\begin{equation}\label{hilbert90_multiplicative}
\begin{array}{rcccccccl}
1\to & \mathbb{F}_{\sqrt{q}}^{\ast} & \stackrel{i}{\to}&  \mathbb{F}_q^{\ast} &
 \to & \mathbb{F}_q^{\ast} & 
 \xrightarrow{\Nm} & \mathbb{F}_{\sqrt{q}}^{\ast} &\to 1 \\
     &  &  &     \alpha &\mapsto &\alpha^{1-\sqrt{q}}&
     & &
\end{array},
\end{equation}
where $i$ means natural inclusion.
\end{lemma}

\begin{corollary}\label{corollary_hilbert90}
\begin{enumerate}[{\rm (i)}]
\item The equation of $X$ with $\beta \in \mathbb{F}_q$
\begin{equation}\label{additive_eq}
X^{\sqrt{q}} -X -\beta =0
\end{equation}
has a solution in $\mathbb{F}_q$
if and only if $\Tr \beta=0$.
In this case, all roots of $(\ref{additive_eq})$ are in $\mathbb{F}_q$.
\item The equation of $X$ with $\beta \in \mathbb{F}_q^{\ast}$
\begin{equation}\label{multiplicative_eq}
X^{\sqrt{q}-1} - \beta =0
\end{equation}
has a solution in $\mathbb{F}_q^{\ast}$
if and only if $\Nm \beta=1$.
In this case, all roots of $(\ref{multiplicative_eq})$ are in $\mathbb{F}_q^{\ast}$.
\end{enumerate}
\end{corollary}
\begin{proof}
(i)
Let $\alpha \in \mathbb{F}_q$ be a root of (\ref{additive_eq}).
Then $\beta = \alpha^{\sqrt{q}} -\alpha$. Hence $\Tr \beta =0$.
Conversely,
if $\Tr \beta =0$, there is an element $\alpha \in \mathbb{F}_q$ such that
$\alpha^{\sqrt{q}} -\alpha =\beta$
by (\ref{hilbert90_addtive}).
When $\alpha$ is a root of (\ref{additive_eq}),
the set of all roots is 
$\{ \alpha + \gamma | \gamma \in \mathbb{F}_{\sqrt{q}}\}$.

(ii) can be proved by using (\ref{hilbert90_multiplicative}).
\end{proof}

The above corollary can be generalized slightly as follows:

\begin{proposition}\label{equation_new}
Let $\alpha \in \mathbb{F}_q^{\ast}$ and $\beta \in \mathbb{F}_q$,
and consider the equation in $X$:
\begin{equation}\label{unifiedequation}
X^{\sqrt{q}} + \alpha X + \beta =0.
\end{equation}
\begin{enumerate}[{\rm(i)}]
 \item When $\Nm \alpha =1$, 
$(\ref{unifiedequation})$ has a solution in $\mathbb{F}_q$ if and only if
$\beta =0$ or $\alpha = \beta^{1-\sqrt{q}}$.
In this case, all roots of $(\ref{unifiedequation})$
are $\mathbb{F}_q$-solutions.
 \item If $\Nm \alpha \neq 1$, then $(\ref{unifiedequation})$
 has a unique $\mathbb{F}_q$-solution.
\end{enumerate}
\end{proposition}
\begin{proof}
For a fixed $\alpha \in \mathbb{F}_q^{\ast}$,
consider the $\mathbb{F}_{\sqrt{q}}$-linear map
$\varphi_{\alpha} : \mathbb{F}_q \to \mathbb{F}_q$
given by $\varphi_{\alpha}(z) = z^{\sqrt{q}} + \alpha z$.
Since $\Nm (-\alpha) = \Nm \alpha$,
the kernel of $\varphi_{\alpha}$ is nontrivial if and only if
$\Nm \alpha =1$ by Corollary~\ref{corollary_hilbert90}~(ii).

Therefore if $\Nm \alpha \neq 1$, then $\varphi_{\alpha}$
is an isomorphism. Hence
there is a unique element of $\mathbb{F}_q$, say $z_0$,
so that $\varphi_{\alpha}(z_0)= - \beta$.
So the proof of (ii) is done.

Next suppose that $\Nm \alpha =1$.
In this case, $\dim \Ker \varphi_{\alpha} >0$
from the first paragraph of this proof,
in particular there is an element $z_0 \in \mathbb{F}_q^{\ast}$
such that $z_0^{\sqrt{q}-1} + \alpha =0$.
Hence
\[
\Ker \varphi_{\alpha} = 
\{ \zeta z_0 \mid \zeta \in \mathbb{F}_{\sqrt{q}}^{\ast}\} \cup \{0\}.
\]
If (\ref{unifiedequation}) has a solution $z_1$,
the set of all solutions of (\ref{unifiedequation}) is
\[
\{ z_1 + \zeta z_0 \mid \zeta \in \mathbb{F}_{\sqrt{q}}^{\ast}\} \cup \{z_1\},
\]
which is a subset of $\mathbb{F}_q$.
So the additional statement of (i) is done.

Now we prove the first part of (i).
To prove the `only if' part,
we may suppose that $\beta \neq 0$.
Let $z_1$ be an $\mathbb{F}_q$-solution of (\ref{unifiedequation}).
Then
$\displaystyle \alpha = -\frac{z_1^{\sqrt{q}}+\beta}{z_1}$.
Since $\Nm \alpha =1$ and $(z_1^{\sqrt{q}})^{\sqrt{q}}= z_1$,
we have
\begin{eqnarray*}
1 &=& 
\Nm \alpha = 
-(\frac{z_1^{\sqrt{q}}+\beta}{z_1})^{\sqrt{q}}
(-\frac{z_1^{\sqrt{q}}+\beta}{z_1})\\
&=& \frac{z_1^{\sqrt{q}+1}+ \beta^{\sqrt{q}}z_1^{\sqrt{q}} +\beta z_1 
+\beta^{\sqrt{q}+1}}{z_1^{\sqrt{q}+1}}.
\end{eqnarray*}
Hence
\[
z_1^{\sqrt{q}} + \beta^{1- \sqrt{q}}z_1  + \beta =0.
\]
Since $z_1$ satisfies (\ref{unifiedequation}),
we get $\alpha = \beta^{1-\sqrt{q}}$.
Conversely, if $\beta =0$, obviously (\ref{unifiedequation})
has an $\mathbb{F}_q$-solution.
Hence we may consider the case $\alpha = \beta^{1-\sqrt{q}}$
with $\beta \neq 0$.
In this case,
(\ref{unifiedequation}) is
\begin{equation}\label{prefinalequation}
X^{\sqrt{q}} + \beta^{1-\sqrt{q}} X + \beta =0
\end{equation}
which is equivalent to
\begin{equation}\label{finalequation}
(\beta X)^{\sqrt{q}} + (\beta X) + \beta^{\sqrt{q}+1} =0.
\end{equation}
Since
\[
(-\beta^{\sqrt{q}+1})^{\sqrt{q}-1} = \beta^{q-1} = 1
\]
because $\beta \in \mathbb{F}_q^{\ast}$,
we know $-\beta^{\sqrt{q}+1} \in \mathbb{F}_{\sqrt{q}}^{\ast}$.
Since $\Tr: \mathbb{F}_q \to \mathbb{F}_{\sqrt{q}}$ is surjective,
the equation (\ref{finalequation}) in $(\beta X)$
has a solution 
in $\mathbb{F}_q$,
and also (\ref{prefinalequation}) has.
\end{proof}

\subsection{Proof of Theorem~\ref{maintheorem}}
\begin{lemma}\label{coordinate_simplify}
Let $C$ be a relative of the Hermitian curve with at least
two rational inflexions.
Then we can choose a system of coordinates $x,y,z$
of $\mathbb{P}^2$ over $\mathbb{F}_q$ such that
\begin{itemize}
\item $P=(1,0,0)$ and $Q=(0,1,0)$ are inflexions of $C$, and
\item $T_P(C) =\{ y=0 \}$ and $T_Q(C) =\{ x=0\}$.
\end{itemize}
\end{lemma}
\begin{proof}
Since
$C.T_P(C) = (\sqrt{q}+1)P$
and 
$C.T_Q(C) = (\sqrt{q}+1)Q$,
three $\mathbb{F}_q$-lines
$T_P(C)$, $T_Q(C)$ and $PQ$ are not concurrent,
where $PQ$ denotes the line through $P$ and $Q$.
Therefore, we may choose coordinates as
$T_P(C) =\{ y=0 \}$, $T_Q(C) =\{ x=0\}$
and $PQ=\{z=0\}$.
\end{proof}

\begin{lemma}\label{equation_simplified}
Under the situation described in Lemma~\rm{\ref{coordinate_simplify}},
the curve is defined by
$(x^{\sqrt{q}}, y^{\sqrt{q}}, z^{\sqrt{q}}) A \, {}^t\!(x,y,z) 
                                            =0,$
where
\[
A =
\begin{pmatrix}
0& a_{12} & 0 \\
a_{21} & 0 & 0\\
0 & 0 &a_{33}
\end{pmatrix} 
\in PGL(3, \mathbb{F}_q) \ \text{with} \ a_{12}a_{21}a_{33} \neq 0.
\]
In other words,
$C$ is defined by
\begin{equation}\label{simplifiedequation}
 a_{12}x^{\sqrt{q}}y + a_{21} x y^{\sqrt{q}} + a_{33}z^{\sqrt{q}+1} =0.
\end{equation}
\end{lemma}
\begin{proof}
Let $C=C_A$, where $A= (a_{ij}) \in PGL(3, \mathbb{F}_q)$.
Since $(1,0,0)A{}^t\! (1,0,0)=0$ and  $(0,1,0)A{}^t\! (0,1,0)=0$,
$a_{11}= a_{22}=0$.

Note that
the tangent line at $(b_1, b_2, b_3) \in C$ is given by
\[
(b_1^{\sqrt{q}}, b_2^{\sqrt{q}}, b_3^{\sqrt{q}})A {}^t\! (x,y,z) =0.
\]
Hence $a_{13} =0$ because $T_P(C) =\{ y=0 \}$, and
$a_{23}=0$ because  $T_Q(C) =\{ x=0\}$.
Therefore an equation of $C$ is of the form
\[
 a_{12}x^{\sqrt{q}}y + a_{21} x y^{\sqrt{q}} + a_{31}xz^{\sqrt{q}}
+ a_{32}yz^{\sqrt{q}}+ a_{33}z^{\sqrt{q}+1} =0.
\]
Since $P$ is an inflexion of $C$ and $T_P(C) = \{y=0 \}$,
we have  $C.\{y=0\} = (\sqrt{q}+1)P$.
Hence the equation
\[
 a_{31}xz^{\sqrt{q}}+ a_{33}z^{\sqrt{q}+1} =0
\]
has the root at $z=0$
of multiplicity $(\sqrt{q}+1)$.
So $a_{31}=0$.
Similarly,
since
$C.\{x=0\} = (\sqrt{q}+1)Q$, from the equation
$
 a_{32}yz^{\sqrt{q}}+ a_{33}z^{\sqrt{q}+1} =0,
$
we know $a_{32} =0$.
Finally $\det A \neq 0$
implies $a_{12}a_{21}a_{33} \neq 0$.
\end{proof}

\begin{corollary}\label{proofoffirstpart}
After a suitable projective transformation over $\mathbb{F}_q$,
$C$ is defined by
\[
 x^{\sqrt{q}}y + \omega x y^{\sqrt{q}} + z^{\sqrt{q}+1} =0
\]
with $\omega \in \mathbb{F}_q^{\ast}$.
Conversely, a curve defined by this type of equation
is Hermitian-relative with two or more inflexions.
\end{corollary}
\begin{proof}
Since $a_{33} \neq 0$, we may suppose that $a_{33} =1$ in 
(\ref{simplifiedequation}).
If one consider $a_{12} y$ as new $y$,
then the equation becomes
\[
x^{\sqrt{q}}y 
+ a_{21}a_{12}^{-{\sqrt{q}} } x y^{\sqrt{q}} 
+ z^{\sqrt{q}+1} =0.
\]
Put $a_{21}a_{12}^{-{\sqrt{q}} }= \omega.$

To confirm the last assertion is not difficult,
actually $(1,0,0)$ and $(0,1,0)$ are inflexions.
\end{proof}
By Corollary~\ref{proofoffirstpart},
we have done the proof of the first paragraph
of Theorem~\ref{maintheorem}.

Now we investigate the curve (\ref{at_least_two_inflexions}),
hereafter, this curve will be denoted by $C_{\omega}$.

\begin{lemma}\label{number_on_infty_line}
The number of $\mathbb{F}_q$-points on $C_{\omega} \cap \{ z=0 \}$
is $\sqrt{q} +1$ if $\Nm \omega = 1$, and
$2$ if $\Nm \omega \neq 1$.
\end{lemma}
\begin{proof}
For any $\omega$,
$P=(1,0,0)$ and $Q=(0,1,0)$
are rational inflexions on $C_{\omega}$.
Since $\{z=0\} \cap \{y=0\} =\{P\}$,
it is enough to count the number of
$\mathbb{F}_q$-solutions of (\ref{at_least_two_inflexions})
when $y=1$ and $z=0$.
The equation in question is
$x^{\sqrt{q}} + \omega x =0.$
Except for $Q$, the other points of $C_{\omega} \cap \{ z=0 \} \cap \{y=1\}$
correspond to
the roots of $x^{\sqrt{q}-1} + \omega  =0.$
By Corollary~\ref{corollary_hilbert90},
the number of $\mathbb{F}_q$-solutions
of this equation is
\[
\left\{
\begin{array}{cl}
\sqrt{q}-1 & \text{if $\Nm \omega \, (= \Nm (-\omega)) =1 $}\\
 0          &\text{ else.}
\end{array}
\right.
\]
Adding $2$ coming from $\{P, Q\}$ in each cases, we have the desired numbers.
\end{proof}

Next we compute the number of rational points of $C_{\omega}$
outside of $\{z=0\}$.

\begin{lemma}\label{number_on_affine}
The number of $\mathbb{F}_q$-solutions of affine equation
\begin{equation}\label{affine_eq_omega}
 x^{\sqrt{q}}y + \omega x y^{\sqrt{q}} + 1 =0
\end{equation}
is
\[
\left\{
\begin{array}{cl}
\sqrt{q}(q-1) & \text{ if $\omega =1$}\\
0             & \text{ if $\Nm \omega =1$ and $\omega \neq 1$}\\
q-1           & \text{ if $\Nm \omega \neq 1$}.
\end{array}
\right.
\]
\end{lemma}
\begin{proof}
If the affine equation (\ref{affine_eq_omega})
has an $\mathbb{F}_q$-solution $(x_0, y_0)$, 
then $x_0y_0^{\sqrt q}$ is a root of the equation in $T$
\begin{equation}\label{polynomial_eq}
T^{\sqrt{q}} + \omega T +1 =0.
\end{equation}
Conversely if $t_0$ is a root of (\ref{polynomial_eq}),
then 
$(\frac{t_0}{y_0^{\sqrt{q}}}, y_0)$ is a solution of (\ref{affine_eq_omega})
for any $y_0 \in \mathbb{F}_q^{\ast}$.
Therefore
the number of solutions of (\ref{affine_eq_omega})
is 
\[
\{ \text{the number of solutions of (\ref{polynomial_eq})} \} \times (q-1).
\]

Suppose that $\Nm \omega =1$.
Then, by Proposition~\ref{equation_new}~(i),
(\ref{polynomial_eq}) has an $\mathbb{F}_q$-solution if and only if
$\omega =1$. In this case, 
it has exactly $\sqrt{q}$ rational solutions, and
totally (\ref{affine_eq_omega}) has $\sqrt{q}(q-1)$ rational solutions.

Suppose that $\Nm \omega \neq 1$.
Then, from the second part of the same proposition,
(\ref{polynomial_eq}) has a unique $\mathbb{F}_q$-solution,
and then (\ref{affine_eq_omega}) has $(q-1)$ rational solutions.
\end{proof}

The following three lemmas will be used for
showing the part (b) of Theorem~\ref{maintheorem}.

The element of $PGL(3, \mathbb{F}_q)$ which comes from
the diagonal matrix with entries $a, b, c$
from the upper left to the lower right
will be denoted by $\diag[a,b.c]$.

\begin{lemma}\label{b_to_diagonal}
For $\eta \in \mathbb{F}_q^{\ast}$,
we have
\[
\begin{pmatrix}
0& 1 & 0 \\
1& 0 & 0 \\
0& 0 & \eta 
\end{pmatrix}
\sim
\diag[1,1,\eta].
\]
\end{lemma}
\begin{proof}
Choose $u \in \mathbb{F}_q$
so that $\Tr u =1$
and $\alpha \in  \mathbb{F}_q^{\ast}$
so that $\Nm \alpha =-1$.
Let
$
T_0 = \begin{pmatrix}
      u & \alpha u^{\sqrt{q}} & 0\\
      1 & -\alpha& 0\\
      0 & 0& 1
      \end{pmatrix}.
$
Then by straightforward computation, we know that
$
T_0^{\ast} \begin{pmatrix}
     0& 1 & 0\\
     1& 0 & 0\\
     0& 0 & \omega
   \end{pmatrix}
   T_0 = \diag[1,1,\eta].
$
\end{proof}

\begin{lemma}\label{number_point_b}
Let $\eta \in \mathbb{F}_q^{\ast} \setminus \mathbb{F}_{\sqrt{q}}^{\ast}$.
Then the number of rational points of $C_{\diag[1,1,\eta]}$
is $\sqrt{q} +1$, and these rational points are on the line $\{z=0\}$.
\end{lemma}
\begin{proof}
The curve $C_{\diag[1,1,\eta]}$ is defined by
\begin{equation}\label{diag_eq_omega}
x^{\sqrt{q}+1} + y^{\sqrt{q}+1} +\eta z^{\sqrt{q}+1} =0.
\end{equation}
If $(x_0, y_0, z_0) \in (\mathbb{F}_q)^3$
is a solution of (\ref{diag_eq_omega}),
then it is a solution of the simultaneous equations:
\[ \left\{
\begin{aligned}
x^{\sqrt{q}+1} + y^{\sqrt{q}+1} =0\\
z^{\sqrt{q}+1} =0
\end{aligned}
\right.
\]
because 
$x_0^{\sqrt{q}+1},  y_0^{\sqrt{q}+1}, z_0^{\sqrt{q}+1}$
 are elements of $\mathbb{F}_{\sqrt{q}}$, 
and $\{1, \eta \}$ are linearly independent over $\mathbb{F}_{\sqrt{q}}$.
Therefore
\[
C_{\diag[1,1,\eta]}(\mathbb{F}_q) =
\{ (\zeta , 1, 0) 
| \Nm \zeta = -1
\}
\subset \{ z=0 \}.
\]
This completes the proof.
\end{proof}

\begin{lemma}\label{equiv_type1}
Let 
$\eta, \eta' \in \mathbb{F}_q^{\ast} \setminus \mathbb{F}_{\sqrt{q}}^{\ast}$.
Then
there is an element $T \in PGL(3, \mathbb{F}_q)$ such that
\[
T^{\ast} \diag[1,1, \eta] T = \diag[1,1,\eta']
\]
if and only if $\eta' \eta^{-1} \in \mathbb{F}_{\sqrt{q}}^{\ast}$.
\end{lemma}
\begin{proof}
Let $\lambda = \eta' \eta^{-1}$, and
suppose that $\lambda \in \mathbb{F}_{\sqrt{q}}^{\ast}$.
Since $\Nm$ is surjective,
there is an element $\alpha \in \mathbb{F}_q^{\ast}$
such that $\alpha^{\sqrt{q}+1} = \lambda$.
So
\[
\diag[1,1, \alpha]^{\ast} \diag[1,1, \eta] \diag[1,1,\alpha]
 = \diag[1,1, \eta'].
\]

Conversely, let $T = (a_{ij}) \in GL(3, \mathbb{F}_q)$
such that
\[
T^{\ast} \diag[1,1, \eta] T
 = \rho \diag[1,1, \eta'],
\]
where we are regarding matrices as in $GL(3, \mathbb{F}_q)$
and $\rho \in \mathbb{F}_q^{\ast}$.
For both curves
$C_{\diag[1,1,\eta]}$ and $C_{\diag[1,1,\eta']}$,
their sets of $\mathbb{F}_q$-points
consist of $\sqrt{q} +1$ points on the line $\{ z=0\}$
by Lemma~\ref{number_point_b}.
Hence the line $\{ z=0\}$
is stable by the linear transformation of $\mathbb{P}^2$
corresponding $T$, that is
\[
\begin{pmatrix}
a_{11}& a_{12} & a_{13} \\
a_{21}& a_{22} & a_{23} \\
a_{31}& a_{32} & a_{33} 
\end{pmatrix}
\begin{pmatrix}
\ast \\ \ast \\ 0
\end{pmatrix}
=
\begin{pmatrix}
\ast \\ \ast \\ 0
\end{pmatrix}.
\]
Hence $a_{31}= a_{32}=0$,
and
\[
\begin{pmatrix}
a_{11}^{\sqrt{q}}& a_{21}^{\sqrt{q}} & 0 \\
a_{12}^{\sqrt{q}}& a_{22}^{\sqrt{q}} & 0 \\
a_{13}^{\sqrt{q}}& a_{23}^{\sqrt{q}} & a_{33}^{\sqrt{q}} 
\end{pmatrix}
\begin{pmatrix}
1& 0 & 0 \\
0& 1 & 0 \\
0& 0 & \eta 
\end{pmatrix}
\begin{pmatrix}
a_{11}& a_{12} & a_{13} \\
a_{21}& a_{22} & a_{23} \\
0& 0 & a_{33} 
\end{pmatrix}
=
\rho
\begin{pmatrix}
1& 0 & 0 \\
0& 1 & 0 \\
0& 0 & \eta' 
\end{pmatrix}.
\]
Comparing $(1,1)$-component of the both sides,
we have 
$a_{11}^{\sqrt{q}+1}+ a_{21}^{\sqrt{q}+1} = \rho$,
which guarantees that $\rho$ is an element of $\mathbb{F}_{\sqrt{q}}^{\ast}$.
Taking determinant of the both sides,
we have
\[
\det T^{(\sqrt{q})}\cdot \eta \cdot \det T  = \rho^3 \eta'.
\]
Hence
$\eta' \eta^{-1} = (\det T)^{\sqrt{q}+1} \rho^{-3} 
\in\mathbb{F}_{\sqrt{q}}^{\ast} $
\end{proof}

\begin{proof}[Proof of Theorem~\ref{maintheorem}]
(a) Suppose that $\omega =1$.
Then by Lemmas~\ref{number_on_infty_line}, \ref{number_on_affine},
the curve $C_{\omega}$
has $\sqrt{q}^3+1$ rational points.
Hence it is Hermitian if $q \neq 4$ \cite{hir-sto-tha-vol1991},
and all rational points on a Hermitian curve are inflexions.
When $q=4$, a non-Hermitian curve with $9$ rational points
can't have inflexions \cite[Remark 3.10]{hom2024}.

(b) Suppose that $\Nm \omega = 1$ and $\omega \neq 1$.
Then by Lemmas \ref{number_on_infty_line}, \ref{number_on_affine},
the curve $C_{\omega}$
has $\sqrt{q}+1$ rational points,
and those points lie on the line $\{ z=0 \}$.
Let $R$ be one of those points.
If $R$ is not an inflexion,
then $C_{\omega}.T_R(C_{\omega}) = \sqrt{q}R + R'$
with $R' \neq R$.
Since $R'$ is $\mathbb{F}_q$-point too,
$R'$ lies on the line $\{ z=0 \}$.
Hence $T_R(C_{\omega})$ coincides with the line $\{ z=0 \}$,
which is impossible.
Therefore all the $\sqrt{q}+1$ rational points are inflexions.

Next we show that
$C_{\omega}$ with
$\Nm \omega = 1$ and $\omega \neq 1$
is projectively equivalent to
\begin{equation}\label{diagonal_of_type1}
x^{\sqrt{q}+1} + y^{\sqrt{q}+1} + \eta z^{\sqrt{q}+1} =0
\end{equation}
for a certain 
$\eta \in \mathbb{F}_q^{\ast} \setminus \mathbb{F}_{\sqrt{q}}^{\ast}.$
Since $\Nm \omega =1$,
there is an element 
$\beta  \in \mathbb{F}_q^{\ast} \setminus \mathbb{F}_{\sqrt{q}}^{\ast}$
such that $\beta^{1-\sqrt{q}} = \omega$.
Using this $\beta$,
(\ref{at_least_two_inflexions}) can be rewritten as
\begin{equation}\label{rewrite_original}
(\beta x)^{\sqrt{q}}y +(\beta x)y^{\sqrt{q}}
             +\beta^{\sqrt{q}} z^{\sqrt{q}+1} =0.
\end{equation}
Replacing $\beta x$ with the new $x$
and applying Lemma~\ref{b_to_diagonal},
we get the equation
\[
x^{\sqrt{q}+1} + y^{\sqrt{q}+1} + \beta^{\sqrt{q}} z^{\sqrt{q}+1} =0.
\]
Replace $\beta^{\sqrt{q}}$ by $\eta$, then we have
(\ref{diagonal_of_type1}).
Conversely, tracing the reverse process, we know that
the curve defined by (\ref{diagonal_of_type1})
is projectively equivalent to $C_{\eta^{\sqrt{q}-1}}$.

The last statement in (b) is just Lemma~\ref{equiv_type1}.

(c) Suppose that $\Nm \omega \neq 1$.
From Lemmas \ref{number_on_infty_line}, \ref{number_on_affine},
$C_{\omega}(\mathbb{F}_q)$ consists of $q+1$ rational points.
Note that $P=(1,0,0)$ and $Q=(0,1,0)$
are inflexions, and
they are only $\mathbb{F}_q$-points lie on the line
$\{z=0\}$
by Lemma~\ref{coordinate_simplify}.

We will show that any point
$R \in C_{\omega}(\mathbb{F}_q)\setminus \{P, Q\}$
is not an inflexion.
Since $R$ does not lie on the line $\{z=0\}$,
we may assume that $R=(x_0, y_0, 1)$,
where $(x_0, y_0)$ satisfies the affine equation (\ref{affine_eq_omega}),
in particular $x_0 \neq 0$.
We want compute the divisor
$C_{\omega}. T_R(C_{\omega})$.
Let $f(x,y) = x^{\sqrt{q}}y + \omega x y^{\sqrt{q}} + 1$.
Since $f_x = \omega  y^{\sqrt{q}}$
and $f_y =  x^{\sqrt{q}}$,
the tangent line at $R$ is given by
\[
 \omega  y_0^{\sqrt{q}}x +  x_0^{\sqrt{q}}y +1=0,
\]
or, alternatively
\begin{equation}\label{tangent_at_R}
y = -\omega (\frac{y_0}{x_0})^{\sqrt{q}} x - (\frac{1}{x_0})^{\sqrt{q}}
\end{equation}
because $x_0 \neq 0$.
Take $\sqrt{q}$-th power of (\ref{tangent_at_R}), then
\begin{equation}\label{qthpower_tangent_at_R}
y^{\sqrt{q}} = -\omega^{\sqrt{q}} (\frac{y_0}{x_0}) x^{\sqrt{q}} 
- (\frac{1}{x_0}).
\end{equation}
Substitute (\ref{tangent_at_R}) and (\ref{qthpower_tangent_at_R})
into (\ref{affine_eq_omega}),
then
\begin{equation}\label{substitute}
-\left\{
\left( \omega (\frac{y_0}{x_0})^{\sqrt{q}} + \omega^{\sqrt{q}+1}\frac{y_0}{x_0}
\right)
x^{\sqrt{q}+1} +
(\frac{1}{x_0})^{\sqrt{q}}x^{\sqrt{q}}
+\frac{\omega}{x_0}x -1
\right\}=0.
\end{equation}
On the other hand, since
\[
0 = \omega \left(
\frac{f(x_0, y_0)}{x_0^{\sqrt{q}+1}}
\right)^{\sqrt{q}}
=  \omega (\frac{y_0}{x_0})^{\sqrt{q}} + \omega^{\sqrt{q}+1}\frac{y_0}{x_0} 
  + \frac{\omega}{x_0^{\sqrt{q}+1}},
\]
the coefficient of $x^{\sqrt{q} +1}$ in (\ref{substitute})
is $\frac{\omega}{x_0^{\sqrt{q}+1}}$.
Hence (\ref{substitute})
can be written as
\begin{align}
&\frac{\omega}{x_0^{\sqrt{q}+1}} x^{\sqrt{q}+1}
-\frac{1}{x_0^{\sqrt{q}}} x^{\sqrt{q}}
-\frac{\omega}{x_0}x  +1 \\
=& \frac{\omega}{x_0^{\sqrt{q}+1}}
\left(x-x_0 \right)^{\sqrt{q}}
\left( x - \frac{x_0}{\omega}\right)=0.
\end{align}
Since $\omega \neq 1$, $x_0 \neq \frac{x_0}{\omega}$,
which means that $R$ is not an inflexion.
Hence only $P$ and $Q$ are inflexions of $C_{\omega}$.

Lastly we classify the curves $C_{\omega}$ with $\omega \neq 1$ 
up to projective equivalence.
For $\omega, \omega' \in \mathbb{F}_q^{\ast}\setminus \Ker( \Nm)$,
let
\[
A_{\omega} = 
\begin{pmatrix}
0&1&0\\
\omega&0&0\\
0&0&1
\end{pmatrix}
\text{ and }
A_{\omega'} = 
\begin{pmatrix}
0&1&0\\
\omega'&0&0\\
0&0&1
\end{pmatrix}.
\]
Then two curves $C_{\omega}$ and $C_{\omega'}$ are projectively equivalent 
to each other
if and only if there is a matrix
$T = (a_{ij}) \in GL(3, \mathbb{F}_q)$
such that 
\begin{equation}\label{last_condition_c}
T^{\ast} A_{\omega} T = \rho A_{\omega'}
\end{equation}
for some $\rho \in \mathbb{F}_q^{\ast}$.
Since inflexions of both curves are $\{P, Q\}$ that are lie on
the line $\{z=0\}$, and
\[
T_P(C_{\omega})\cap T_Q(C_{\omega}) =
T_P(C_{\omega'})\cap T_Q(C_{\omega'})
= \{ (0,0,1)\},
\]
the line $\{ z=0 \}$
and the point $(0,0,1)$ are stable by the projective transformation $T$.
Hence we may choose $T$ as
$
\begin{pmatrix}
a_{11}&a_{12}&0\\
a_{21}&a_{22}&0\\
0&0&1
\end{pmatrix}.
$
Then, comparing $(3,3)$-entries of both sides of (\ref{last_condition_c}),
we know $\rho =1$.
Now we write down (\ref{last_condition_c}) componentwise:
\begin{align}
\omega a_{21}^{\sqrt{q}}a_{11} +  a_{11}^{\sqrt{q}}a_{21} &= 0 \label{1-1}\\
\omega a_{22}^{\sqrt{q}}a_{12} +  a_{12}^{\sqrt{q}}a_{22} &= 0 \label{2-2}\\
\omega a_{21}^{\sqrt{q}}a_{12} +  a_{11}^{\sqrt{q}}a_{22} &= 1 \label{1-2}\\
\omega a_{22}^{\sqrt{q}}a_{11} +  a_{12}^{\sqrt{q}}a_{21} &= \omega' 
\label{2-1},
\end{align}
where (\ref{1-1})--(\ref{2-1}) come from
comparison of $(1,1)$,$(2,2)$, $(1,2)$ and $(2,1)$ components respectively.
Take $\sqrt{q}$th power of (\ref{1-1})
\begin{equation}\label{qth_power_1-1}
\omega^{\sqrt{q}} a_{21}a_{11}^{\sqrt{q}} +  a_{11}a_{21}^{\sqrt{q}} = 0,
\end{equation}
and make $(\ref{qth_power_1-1}) \times \omega - (\ref{1-1})$,
then
\[
(\omega^{\sqrt{q}+1} -1) a_{11}^{\sqrt{q}}a_{21} =0.
\]
Since $\Nm \omega \neq 1$, we conclude that $a_{11}=0$ or $a_{21}=0$.
Similarly, using (\ref{2-2}), we have $a_{22}=0$ or $a_{12}=0$.
Since $\det T \neq 0$,
only possibilities of $T$ are
\[
T = \begin{pmatrix}
a_{11}&0&0\\
0&a_{22}&0\\
0&0&1
\end{pmatrix}
\text{or}
\begin{pmatrix}
0&a_{12}&0\\
a_{21}&0&0\\
0&0&1
\end{pmatrix}.
\]
If the first case occurs, then
$a_{11}^{\sqrt{q}}a_{22}=1$ by (\ref{1-2}) and
$\omega a_{22}^{\sqrt{q}}a_{11} = \omega'$ by (\ref{2-1}),
which implies that $\omega = \omega'$.
If the second case occurs,
then
\[
\omega a_{21}^{\sqrt{q}}a_{12}=1 \text{ and } a_{12}^{\sqrt{q}}a_{21}=\omega',
\]
which implies that $\omega \omega'^{\sqrt{q}} =1$.
\end{proof}
\gyokan

\noindent
\textbf{Acknowledgment.}
The second author was partially supported by
the National Research Foundation of Korea(NRF) grant funded by the Korea government(MSIT) (2022R1A2C1012291).


\end{document}